\documentclass{article}
\usepackage{latexsym,amsthm,amssymb,amscd,amsmath}
\newcounter{alphthm}
\newtheorem{thm}{Theorem}

\newtheorem{defn}{Definition}
\newtheorem{prop}{Proposition}
\newtheorem{cor}{Corollary}
\newtheorem{lem}{Lemma}

\newcommand{\be}{\begin{equation}}
\newcommand{\ee}{\end{equation}}
\newcommand{\ben}{\begin{enumerate}}
\newcommand{\een}{\end{enumerate}}
\newcommand{\pa}{{\partial}}

\newcommand{\g}{{\bf g}}
\newcommand{\pxi}{{\pa \over \pa x^i}}
\newcommand{\D}{\delta^t}

\def\beq{\begin{equation}}
\def\eeq{\end{equation}}

\title{On Twisted Products Finsler Manifolds}
\author{E. Peyghan, A. Tayebi and L. Nourmohammadi Far}
\begin{document}


\maketitle
\begin{abstract}
On the product of two Finsler manifolds $M_1\times M_2$, we consider the twisted metric $F$ which is construct by using
Finsler metrics $F_1$ and $F_2$ on the manifolds $M_1$ and $M_2$,
respectively. We introduce horizontal and vertical distributions on twisted  product Finsler manifold and  study   C-reducible and  semi-C-reducible  properties of this manifold.   Then we obtain the Riemannian curvature and some of non-Riemannian  curvatures of the twisted product Finsler manifold such as  Berwald curvature, mean Berwald curvature and we find the relations between   these objects and their corresponding objects on $M_1$ and $M_2$. Finally, we study locally dually flat twisted product Finsler manifold.
\end{abstract}

\section{Introduction.}
Twisted and warped product structures are widely used in geometry to construct
newexamples of semi-Riemannian manifolds with interesting curvature properties.
(see \cite{BEE}\cite{GM}\cite{Koz}\cite{O}\cite{PTN}). Twisted product metric tensors, as a generalization
of warped product metric tensors, have also been useful in the study of several
aspects of submanifold theory, namely, in hypersurfaces of complex space forms
\cite{LR}, in Lagrangian submanifolds \cite{CDVA} and in decomposition of curvature netted hypersurfaces \cite{K},
etc.

The notion of twisted product of Riemannian manifolds was mentioned first by Chen in \cite{C}, and was generalized for the pseudo-Riemannian case by Ponge and Reckziegel \cite{PR}. Chen extended the study of twisted product for CR-submanifolds in K\"{a}hler manifolds \cite{C1}.

On the other hand, Finsler geometry is a natural extension of  Riemannian geometry without the quadratic restriction. Therefore, it is  natural to extending the construction of  twisted product manifolds for Finsler geometry. In \cite{Koz}, Kozma-Peter-Shimada extended the construction of twisted product for  the Finsler geometry.

Let $(M_1,F_1)$ and $(M_2,F_2)$ are two Finsler manifolds with Finsler metrics $F_1$ and $F_2$, respectively, and
$f:M_1\times M_2\longrightarrow R^+$ be a smooth function. On the product manifold $M_1\times M_2$, we consider the metric
\[
F(v_1,v_2)=\sqrt{F_1^2(v_1)+f^2(x,y)F^2_2(v_2)}
\]
for all $(x,y)\in M_1\times M_2$ and $(v_1,v_2)\in TM^{\circ}_1\times TM^{\circ}_2$, where $TM^{\circ}_1$ is the slit tangent manifold $TM^{\circ}_1=TM_1\setminus \{\circ\}$. The manifold $M_1\times M_2$ endowed with this metric, we call the twisted product of the manifolds $M_1$ and $M_2$ and  denote it by $M_1\times_{f} M_2$. The function $f$ will be called the twisted function.
In particular,  if $f$ is constant on $M_2$, then $M_1\times_{f} M_2$ is called warped product manifold.

Let $(M, F)$ be a Finsler manifold.  The second and third order derivatives of ${1\over 2} F_x^2$ at $y\in T_xM_{0}$ are the symmetric trilinear forms ${\bf g}_y$ and ${\bf C}_y$ on $T_xM$, which called the fundamental tensor and   Cartan torsion, respectively.  A Finsler metric is called semi-C-reducible if its Cartan tensor is given by
\[
C_{ijk}={\frac{p}{1+n}}\{h_{ij}I_k+h_{jk}I_i+h_{ki}J_j\}+\frac{q}{C^2}I_iI_jI_k,
\]
where $p=p(x,y)$ and $q=q(x,y)$ are scalar function on $TM$, $h_{ij}$ is the angular metric and $C^2=I^iI_i$ \cite{Mat4}. If $q=0$, then $F$ is called C-reducible Finsler metric and if $p=0$, then $F$ is called $C2$-like metric.

The geodesic curves of a Finsler metric $F$ on a smooth manifold $M$, are determined  by the system of second order differential equations $ \ddot c^i+2G^i(\dot c)=0$, where the local functions $G^i=G^i(x, y)$ are called the  spray coefficients.  $F$ is  called a Berwald metric, if  $G^i$  are quadratic in $y\in T_xM$  for any $x\in M$. Taking a trace of  Berwald curvature yields mean Berwald curvature ${\bf E}$.  Then $F$ is said to be isotropic mean Berwald metric if ${\bf E}=\frac{n+1}{2}cF^{-1}{\bf h}$, where ${\bf h}=h_{ij} dx^i\otimes dx^j$ is the angular metric and $c=c(x)$ is a scalar function on $M$ \cite{NST2}.

The second variation of geodesics gives rise to a family of linear maps ${\bf R}_y= R^i_{\ k}  dx^k \otimes \pxi|_x :
T_xM \to T_xM$ at any point $y\in T_xM$.  $R_y$ is called the Riemann curvature in the direction $y$. A Finsler metric $F$ is said to be of scalar flag curvature, if for some scalar function ${\bf K}$ on $TM_0$ the Riemann curvature is in the form $R^i_{\ k}= {\bf K}F^2h^i_j$. If ${\bf K}=constant$, then $F$ is said to be of  constant flag curvature.

In this paper, we introduce the horizontal and vertical
distributions on tangent bundle of a doubly warped product Finsler
manifold and construct the Finsler connection on this manifold.
Then, we study some geometric properties of this product manifold
such as C-reducible and semi-C-reducible. Then, we introduce the Riemmanian curvature of twisted product
Finsler manifold $(M_1\times_{f}M_2,F)$ and  find the relation
between it and Riemmanian curvatures of its components $(M_1,F_1)$
and $(M_2,F_2)$. In the cases that $(M_1\times_{f}M_2,F)$ is flat
or it has the scalar flag curvature, we obtain some results on its
components. Then, we study twisted product Finsler metrics with
vanishing Berwald curvature and isotropic mean Berwald curvature,
respectively. Finally, we  study  locally dually flat twisted product Finsler manifold. We prove that there is not exist any locally dually flat proper twisted product Finsler manifold.
\section{Preliminary}
Let $M$ be an $n$-dimensional $ C^\infty$ manifold. Denote by $T_x M $ the tangent space at $x \in M$,  by $TM=\cup _{x \in M} T_x M $ the tangent bundle of $M$, and by $TM^{\circ} = TM \setminus \{ 0 \}$ the slit tangent bundle on $M$ \cite{PH}. A  Finsler metric on $M$ is a function $ F:TM \rightarrow [0,\infty)$ which has the following properties:\\
(i) $F$ is $C^\infty$ on $TM^{\circ}$;\\
(ii) $F$ is positively 1-homogeneous on the fibers of tangent bundle $TM$;\\
(iii) for each $y\in T_xM$, the following quadratic form ${\bf g}_y$ on
$T_xM$  is positive definite,
\[
{\bf g}_{y}(u,v):={1 \over 2} \frac{\partial^2}{\partial s \partial t}\left[  F^2 (y+su+tv)\right]|_{s,t=0}, \ \
u,v\in T_xM.
\]
Let  $x\in M$ and $F_x:=F|_{T_xM}$.  To measure the non-Euclidean feature of $F_x$, define ${\bf C}_y:T_xM\otimes T_xM\otimes T_xM\rightarrow \mathbb{R}$ by
\[
{\bf C}_{y}(u,v,w):={1 \over 2} \frac{d}{dt}\left[{\bf g}_{y+tw}(u,v)
\right]|_{t=0}, \ \ u,v,w\in T_xM.
\]
The family ${\bf C}:=\{{\bf C}_y\}_{y\in TM^{\circ}}$  is called the Cartan torsion. It is well known that ${\bf{C}}=0$ if and only if $F$ is Riemannian \cite{ShDiff}.

For $y\in T_x M^{\circ}$, define  mean Cartan torsion ${\bf I}_y$ by ${\bf I}_y(u):=I_i(y)u^i$, where $I_i:=g^{jk}C_{ijk}$, $C_{ijk}=\frac{1}{2}\frac{\pa g_{ij}}{\pa y^k}$ and $u=u^i\frac{\partial}{\partial x^i}|_x$. By Deicke's  Theorem, $F$ is Riemannian  if and only if ${\bf I}_y=0$.

\bigskip

Let $(M, F)$ be a Finsler manifold. For   $y \in T_xM^{\circ}$, define the  Matsumoto torsion ${\bf M}_y:T_xM\otimes T_xM \otimes T_xM \rightarrow \mathbb{R}$ by ${\bf M}_y(u,v,w):=M_{ijk}(y)u^iv^jw^k$ where
\[
M_{ijk}:=C_{ijk} - {1\over n+1}  \{ I_i h_{jk} + I_j h_{ik} + I_k h_{ij} \},\label{Matsumoto}
\]
$h_{ij}:=FF_{y^iy^j}$ is the angular metric. In \cite{MH}), it is proved that a  Finsler metric $F$ on a manifold $M$ of dimension $n\geq 3$ is a Randers  metric if and only if\  ${\bf M}_y =0$, $\forall y\in TM_0$. A Randers metric $F=\alpha+\beta$  on a manifold $M$ is just a Riemannian metric $\alpha=\sqrt{a_{ij}y^iy^j}$ perturbated by a one form $\beta=b_i(x)y^i$ on $M$ such that  $\|\beta\|_{\alpha}<1$.

A Finsler metric is called semi-C-reducible if its Cartan tensor is given by
\[
C_{ijk}={\frac{p}{1+n}}\{h_{ij}I_k+h_{jk}I_i+h_{ki}I_j\}+\frac{q}{C^2}I_iI_jI_k,
\]
where $p=p(x,y)$ and $q=q(x,y)$ are scalar function on $TM$ and $C^2=I^iI_i$ with $p+q=1$.     In \cite{Mat4}, Matsumoto-Shibata proved that every $(\alpha,\beta)$-metric on a manifold $M$ of dimension $n\geq 3$ is semi-C-reducible.

\bigskip

Given a Finsler manifold $(M,F)$, then a global vector field ${\bf G}$ is induced by $F$ on $TM^{\circ}$, which in a standard coordinate $(x^i,y^i)$ for $TM^{\circ}$ is given by ${\bf G}=y^i {{\partial} \over {\partial x^i}}-2G^i(x,y){{\partial} \over {\partial y^i}}$, where \[
G^i:=\frac{1}{4}g^{il}\Big[\frac{\partial^2 F^2}{\partial x^k
\partial y^l}y^k-\frac{\partial F^2}{\partial x^l}\Big],\ \ \ y\in T_xM.
\]
{\bf G} is called the  spray associated  to $(M,F)$.  In local coordinates, a curve $c(t)$ is a geodesic if and only if its coordinates $(c^i(t))$ satisfy $ \ddot c^i+2G^i(\dot c)=0$ \cite{BCS}.

A Finsler metric $F=F(x,y)$ on a manifold $M$ is said to be locally dually flat if at any point there is a  coordinate system
$(x^i)$ in which the spray coefficients are in the
following form
\[
G^i = -\frac{1}{2}g^{ij}H_{y^j},
\]
where $H=H(x, y)$ is a $C^\infty$ scalar function on $TM^{\circ}$ satisfying $H(x, \lambda y)=\lambda^3H(x, y)$ for all $\lambda > 0$. Such a coordinate system is called an adapted coordinate system. In \cite{shenLec}, Shen proved that  the Finsler metric $F$ on an open subset $U\subset \mathbb{R}^n$ is dually flat if and
only if it satisfies  $(F^2)_{x^ky^l}y^k=2(F^2)_{x^l}$.

\bigskip

For a tangent vector $y \in T_xM^{\circ}$, define ${\bf B}_y:T_xM\otimes T_xM \otimes T_xM\rightarrow T_xM$ and ${\bf E}_y:T_xM \otimes T_xM\rightarrow \mathbb{R}$ by ${\bf B}_y(u, v, w):=B^i_{\ jkl}(y)u^jv^kw^l{{\partial } \over {\partial x^i}}|_x$ and ${\bf E}_y(u,v):=E_{jk}(y)u^jv^k$
where
\[
B^i_{\ jkl}:={{\partial^3 G^i} \over {\partial y^j \partial y^k \partial y^l}},\ \ \ E_{jk}:=\frac{1}{2}B^m_{\ jkm}.
\]
$\bf B$ and $\bf E$ are called the Berwald curvature and mean Berwald curvature, respectively.  Then $F$ is called a Berwald metric and weakly Berwald metric if ${\bf{B}}=0$ and ${\bf{E}}=0$, respectively  \cite{ShDiff}. It is  proved that on  a Berwald space,  the parallel translation along any geodesic preserves the Minkowski functionals \cite{Ich}.

A Finsler metric $F$ is said to be isotropic Berwald metric and isotropic mean Berwald metric if its Berwald curvature and mean Berwald curvature is in the following form, respectively
\begin{eqnarray}\label{IBCurve}
&&B^i_{\ jkl}=c\{F_{y^jy^k}\delta^i_{\ l}+F_{y^ky^l}\delta^i_{\ j}+F_{y^ly^j}\delta^i_{\ k}+F_{y^jy^ky^l}y^i\},\\
&&E_{ij}=\frac{1}{2}(n+1)cF^{-1}h_{ij},
\end{eqnarray}
where  $c=c(x)$ is a  scalar function  on $M$ \cite{ChSh}\cite{TN}.

\bigskip

The Riemann curvature ${\bf R}_y= R^i_{\ k}  dx^k \otimes \pxi|_x :
T_xM \to T_xM$ is a family of linear maps on tangent spaces, defined
by
\begin{equation}\label{TP4}
R^i_{\ k} = 2 {\pa G^i\over \pa x^k}-y^j{\pa^2 G^i\over \pa
x^j\pa y^k} +2G^j {\pa^2 G^i \over \pa y^j \pa y^k} - {\pa G^i \over
\pa y^j} {\pa G^j \over \pa y^k}.
\end{equation}
The flag curvature in Finsler geometry is a natural extension of the sectional curvature in Riemannian geometry was  first introduced by L. Berwald \cite{Be}. For a flag $P={\rm span}\{y, u\} \subset T_xM$ with flagpole $y$, the  flag curvature ${\bf K}={\bf K}(P, y)$ is defined by
\begin{equation}\label{TP5}
{\bf K}(P, y):= {\g_y (u, {\bf R}_y(u)) \over \g_y(y, y) \g_y(u,u)
-\g_y(y, u)^2 }.
\end{equation}
We say that a Finsler metric $F$ is   of scalar curvature if for any $y\in T_xM$, the flag curvature ${\bf K}= {\bf K}(x, y)$ is a scalar function on the slit tangent bundle $TM^{\circ}$. If ${\bf K}=constant$, then $F$ is said to be of  constant flag curvature.


\section{Nonlinear Connection}
Let $(M_1,F_1)$ and $(M_2,F_2)$ be two Finsler manifolds. Then the functions
\begin{equation}
(i)\ g_{ij}(x,y)=\frac{1}{2}\frac{\partial^2F_1^2(x,y)}{\partial y^i\partial y^j},\ \ \ (ii)\ g_{\alpha\beta}(u,v)=\frac{1}{2}\frac{\partial^2F_2^2(u,v)}
{\partial v^\alpha\partial v^\beta},\label{metr}
\end{equation}
define a Finsler tensor field of type $(0,2)$ on $TM^\circ_1$ and $TM^\circ_2$, respectively.
 Now let $(M_1\times{}_{f}M_2,F)$ be a doubly warped Finsler manifold, $\textbf{x}=(x,u)\in M$, $\textbf{y}=(y,v)\in T_\textbf{x}M$, $M=M_1\times M_2$ and $T_\textbf{x}M=T_xM_1\oplus T_uM_2$. Then by using (\ref{metr}) we conclude that
\begin{equation}
\Big(\textbf{g}_{ab}(x,u,y,v)\Big)=\Big(\frac{1}{2}\frac{\partial^2F^2(x,u,y,v)}{\partial \textbf{y}^a\textbf{y}^b}\Big)=\left[
\begin{array}{cc}
g_{ij}&0\\
0&f^2g_{\alpha\beta}
\end{array}
\right],\label{Mat}
\end{equation}
where $\textbf{y}^a=(y^i,v^\alpha)$, $\textbf{g}_{ij}=g_{ij}$,
$\textbf{g}_{\alpha\beta}=f^2g_{\alpha\beta}$,
$\textbf{g}_{i\beta}=\textbf{g}_{\alpha j}=0$, $i,
j,\ldots\in\{1,\ldots,n_1\}$, $\alpha,
\beta,\ldots\in\{1,\ldots,n_2\}$ and $a,
b,\ldots\in\{1,\ldots,n_1+n_2\}$.

Now we consider the the spray coefficients of $F_1$, $F_2$ and $F$ as
\begin{eqnarray}
G^i(x,y)\!\!\!\!&=&\!\!\!\!\frac{1}{4}g^{ih}\Big(\frac{\partial^2F_1^2}{\partial y^h\partial x^j}y^j-\frac{\partial F_1^2}{\partial x^h}\Big)(x,y),\label{spray1}\\
G^{\alpha}(u,v)\!\!\!\!&=&\!\!\!\!\frac{1}{4}g^{\alpha\gamma}\Big(\frac{\partial^2F_2^2}{\partial v^\gamma\partial u^\beta}v^\beta
-\frac{\partial F_2^2}{\partial u^\gamma}\Big)(u,v),\label{spray2}\\
{\textbf{G}}^a(\textbf{x},\textbf{y})\!\!\!\!&=&\!\!\!\!\frac{1}{4}{\textbf{g}}^{ab}\Big(\frac{\partial^2F^2}{\partial \textbf{y}^b\partial \textbf{x}^c}\textbf{y}^c
-\frac{\partial F^2}{\partial \textbf{x}^b}\Big)(\textbf{x},\textbf{y}).\label{spray3}
\end{eqnarray}
Taking into account the homogeneity of both $F_1^2$ and $F_2^2$, and using (\ref{spray1}) and (\ref{spray2}), we can conclude that $G^i$ and $G^\alpha$ are positively homogeneous of degree two with respect to $(y^i)$ and $(v^\alpha)$,  respectively. Hence from Euler theorem for homogeneous functions,  we infer that
\[
\frac{\partial G^i}{\partial y^j}y^j=2G^i, \ \   \textrm{and} \ \ \frac{\partial G^\alpha}{\partial v^\beta}v^\beta=2G^\alpha.
\]
By setting $a=i$ in (\ref{spray3}) we have
\[
\textbf{G}^i(x,u,y,v)=\frac{1}{4}\textbf{g}^{ih}\Big(\frac{\partial^2F^2}{\partial y^h\partial x^j}y^j+\frac{\partial^2F^2}{\partial y^h\partial u^\alpha}v^\alpha-\frac{\partial F^2}{\partial x^h}\Big).
\]
Direct calculations give us
\begin{eqnarray*}
&&\frac{\partial F^2}{\partial
x^h}=\frac{\partial F_1^2}{\partial x^h}+\frac{\partial
f^2}{\partial x^h}F_2^2,\\
&&\frac{\partial^2F^2}{\partial
y^h\partial x^j}=\frac{\partial^2F_1^2}{\partial y^h\partial
x^j}\\
&&\frac{\partial^2F^2}{\partial y^h\partial u^\alpha}=0.
\end{eqnarray*}
Putting these equations together
$\textbf{g}^{ih}=g^{ih}$ in the above equation
and using (\ref{spray1}) imply that
\begin{equation}
\textbf{G}^i(x,u,y,v)=G^i(x,y)-\frac{1}{2}ff^iF_2^2.\label{spray4}
\end{equation}
Similarly, by setting $a=\alpha$ in (\ref{spray3}) and using (\ref{spray2}) we obtain
\begin{equation}
\textbf{G}^\alpha(x,u,y,v)=G^\alpha(u,v)+f^{-1}(f_jv^\alpha y^j+f_\lambda v^\alpha v^\lambda-\frac{1}{2}f_\gamma g^{\alpha\gamma}F_2^2),\label{spray5}
\end{equation}
where $f_i=\frac{\partial f}{\partial x^i}$, $f_\gamma=\frac{\partial f}{\partial u^\gamma}$, $f^i=g^{ih}f_h$ and $f^\gamma=g^{\lambda\gamma}f_\lambda$. Therefore we have $\textbf{G}^a=(\textbf{G}^i,\textbf{G}^\alpha)$, where $\textbf{G}^a$, $\textbf{G}^i$ and $\textbf{G}^\alpha$ are given by (\ref{spray3}), (\ref{spray4}) and (\ref{spray5}), respectively.

Now, we put
\begin{equation}
(i)\ \ \textbf{G}^a_b:=\frac{\partial \textbf{G}^a}{\partial \textbf{y}^b},\ \ \ (ii)\ \ G^i_j:=\frac{\partial G^i}{\partial y^j},\ \ \
(iii)\ \ G^\alpha_\beta:=\frac{\partial G^\alpha}{\partial v^\beta}.\label{spray6}
\end{equation}
Then we have the following.
\begin{lem}
The coefficients $\textbf{G}^a_b$ defined by (\ref{spray6}) satisfy in the following
\begin{equation}
\Big(\textbf{G}^a_b(x,u,y,v)\Big)=\left[
\begin{array}{cc}
\textbf{G}^i_j(x,u,y,v)&\textbf{G}^\alpha_j(x,u,y,v)\\
\textbf{G}^i_\beta(x,u,y,v)&\textbf{G}^\alpha_\beta(x,u,y,v)
\end{array}
\right],
\end{equation}
where
\begin{eqnarray}
\textbf{G}^i_j(x,u,y,v)\!\!\!\!&:=&\!\!\!\!\frac{\partial\textbf{G}^i}{\partial
y^j}=G^i_j+C^{ih}_jff_hF_2^2
,\label{spray7}\\
\textbf{G}^i_\beta(x,u,y,v)\!\!\!\!&:=&\!\!\!\!\frac{\partial\textbf{G}^i}{\partial
v^\beta}=-ff^iv_\beta,\label{spray8}\\
\textbf{G}^\alpha_j(x,u,y,v)\!\!\!\!&:=&\!\!\!\!\frac{\partial\textbf{G}^\alpha}{\partial y^j}=f^{-1}f_jv^\alpha,\label{spray9}\\
\textbf{G}^\alpha_\beta(x,u,y,v)\!\!\!\!&:=&\!\!\!\!\frac{\partial\textbf{G}^\alpha}{\partial
v^\beta}= G^\alpha_\beta+f^{-1}(C^{\alpha\gamma}_\beta f_\gamma F_2^2+f_jy^j\delta^\alpha_\beta-f^\alpha v_\beta\nonumber\\
\!\!\!\!&&\!\!\!\!\hspace{1.2cm}+f_\beta v^\alpha+f_\gamma
v^\gamma\delta^\alpha_\beta).\label{spray10}
\end{eqnarray}
\end{lem}

\bigskip

Next, $VTM^{\circ}$ kernel of the differential of the projection map
\[
\pi:=(\pi_1, \pi_2):TM_1^{\circ}\oplus TM_2^{\circ}\rightarrow M_1\times M_2,
\]
which is a well-defined subbundle of $TTM^\circ$, is considered. Locally, $\Gamma(VTM^{\circ})$ is spanned by the natural vector fields $\{\frac{\partial}{\partial y^1},\ldots, \frac{\partial}{\partial y^{n_1}}, \frac{\partial}{\partial v^1},\ldots, \frac{\partial}{\partial v^{n_2}}\}$ and it is called the \textit{twisted vertical distribution} on $TM^\circ$. Then, using the functions given by (\ref{spray7})-(\ref{spray10}), the nonholonomic vector fields are defined as following
\begin{eqnarray}
\frac{\D}{\D x^i}:\!\!\!\!&=&\!\!\!\!\frac{\partial}{\partial x^i}-\textbf{G}^j_i\frac{\partial}{\partial y^j}-\textbf{G}^{\beta}_i\frac{\partial}{\partial v^{\beta}},\label{dec1}\\
\frac{\D}{\D u^\alpha}:\!\!\!\!&=&\!\!\!\!\frac{\partial}{\partial
u^\alpha}-\textbf{G}^j_\alpha\frac{\partial}{\partial
y^j}-\textbf{G}_\alpha^{\beta}\frac{\partial}{\partial v^{\beta}},
\end{eqnarray}
which make it possible to construct a complementary vector subbundle $HTM^{\circ}$ to $VTM^{\circ}$ in $TTM^{\circ}$ as follows
\[
HTM^{\circ}:=span\{\frac{\D}{\D x^1},\ldots,\frac{\D}{\D x^{n_1}}, \frac{\D}{\D u^1},\ldots, \frac{\D}{\D u^{n_2}}\}.
\]
$HTM^{\circ}$ is called the \textit{twisted horizontal distribution} on $TM^{\circ}$. Thus the tangent bundle of $TM^\circ$ admits the decomposition
\begin{equation}
TTM^\circ=HTM^\circ\oplus VTM^\circ.\label{dec}
\end{equation}
It is shown that $\textbf{G}:=(\textbf{G}^a_b)$ is a nonlinear connection on $TM=TM_1\oplus TM_2$. In the following, we compute the non-linear connection of a  twisted product Finsler manifold.
\begin{prop}
If $(M_1\times{}_{f}M_2,F)$ is a twisted product Finsler manifold, then $\textbf{G}=(\textbf{G}^a_b)$ is the nonlinear connection on $TM$. Further, we have
\begin{eqnarray*}
&&\frac{\partial \textbf{G}^i_j}{\partial y^k}y^k+\frac{\partial \textbf{G}^i_j}{\partial v^\gamma}v^\gamma=\textbf{G}^i_j,\\
&&\frac{\partial \textbf{G}^i_\beta}{\partial y^k}y^k+\frac{\partial \textbf{G}^i_\beta}{\partial v^\gamma}v^\gamma=\textbf{G}^i_\beta,
\\
&&\frac{\partial \textbf{G}^\alpha_j}{\partial y^k}y^k+\frac{\partial \textbf{G}^\alpha_j}{\partial v^\gamma}v^\gamma=\textbf{G}^\alpha_j,\\
&&\frac{\partial \textbf{G}^\alpha_\beta}{\partial y^k}y^k+\frac{\partial \textbf{G}^\alpha_\beta}{\partial v^\gamma}v^\gamma=\textbf{G}^\alpha_\beta.
\end{eqnarray*}
\end{prop}
\begin{defn}
Using decomposition (\ref{dec}), the twisted vertical morphism $v^t:TTM^\circ\rightarrow VTM^\circ$ is defined by \[
v^t:=\frac{\partial}{\partial y^i}\otimes \delta^ty^i+\frac{\partial}{\partial v^\alpha}\otimes \delta^tv^\alpha,
\]
where
\begin{eqnarray}
&&\delta^ty^i:=dy^i+\textbf{G}^i_jdx^j+\textbf{G}^i_\beta du^\beta,\\ &&\delta^tv^\alpha:=dv^\alpha+\textbf{G}^\alpha_jdx^j+\textbf{G}^\alpha_\beta du^\beta.\label{new}
\end{eqnarray}
\end{defn}
For this projective morphism, the following hold
\[
v^t(\frac{\partial}{\partial y^i})=\frac{\partial}{\partial y^i},\ \ \ v^t(\frac{\partial}{\partial v^\alpha})=\frac{\partial}{\partial v^\alpha},\ \ \ v^t(\frac{\D}{\D x^i})=0,\ \ \ v^t(\frac{\D}{\D u^i})=0.
\]
From the above equations, we conclude that
\[
(v^t)^2=v^t, \ \ \  \textrm{and} \ \ \ \ker(v^t)=HTM^\circ.
\]
This mapping is called the \textit{twisted vertical projective}.
\begin{defn}
Using decomposition (\ref{dec}), the \textit{doubly warped horizontal projective} $h^t:TTM^\circ\rightarrow HTM^\circ$ is defined by \[
h^t=id-v^t
\]
or
\[
h^t=\frac{\D}{\D x^i}\otimes dx^i+\frac{\D}{\D u^\alpha}\otimes du^\alpha.
\]
\end{defn}
For this projective morphism, the following hold
\[
h^t(\frac{\D}{\D x^i})=\frac{\D}{\D x^i},\ \ \ h^t(\frac{\D}{\D u^\alpha})=\frac{\D}{\D u^\alpha},\ \ \ h^t(\frac{\partial}{\partial y^i})=0,\ \ \ h^t(\frac{\partial}{\partial v^\alpha})=0.
\]
Thus we result that
\[
(h^t)^2=h^t, \ \ \  \textrm{and} \ \ \ \ker(h^t)=VTM^\circ.
\]
\begin{defn}
Using decomposition (\ref{dec}), the twisted almost tangent structure $J^t: HTM^\circ\rightarrow VTM^\circ$ is defined by
\[
J^t: \frac{\partial}{\partial y^i}\otimes dx^i+\frac{\partial}{\partial v^\alpha}\otimes du^\alpha,
\]
or
\[
J^t(\frac{\D}{\D x^i})=\frac{\partial}{\partial y^i},\ \ \ J^t(\frac{\D}{\D u^\alpha})=\frac{\partial}{\partial v^\alpha}, \ \ \ J^t(\frac{\partial}{\partial y^i})=J^t(\frac{\partial}{\partial v^\alpha})=0.
\]
\end{defn}
Thus we result that
\[
(J^t)^2=0, \ \ \  \textrm{and}  \ \ \ \ker J^t=Im J^t=VTM^\circ.
\]
Here, we introduce some geometrical objects of twisted product Finsler manifold. In order to simplify the equations, we rewritten the basis of $HTM^\circ$ and $VTM^\circ$ as follows:
\begin{eqnarray*}
&&\frac{\D}{\D \textbf{x}^a}=\frac{\D}{\D x^i}\delta_a^i+\frac{\D}{\D u^\alpha}\delta_a^\alpha,\\
&&\frac{\partial}{\partial \textbf{y}^a}=\frac{\partial}{\partial y^i}\delta_a^i+\frac{\partial}{\partial v^\alpha}\delta_a^\alpha.
\end{eqnarray*}
Thus
\[
TTM^\circ=span\{\frac{\D}{\D \textbf{x}^a},\frac{\partial}{\partial \textbf{y}^a}\}.
\]
The Lie brackets of this basis is given by
\begin{eqnarray}
&&[\frac{\D}{\D \textbf{x}^a}, \frac{\D}{\D \textbf{x}^b}]=\textbf{R}^c_{\ ab}\frac{\partial}{\partial \textbf{y}^c},\\
&&[\frac{\D}{\D \textbf{x}^a}, \frac{\partial}{\partial \textbf{y}^b}]=\textbf{G}^c_{ab}\frac{\partial}{\partial \textbf{y}^c},\\ &&[\frac{\partial}{\partial \textbf{y}^a}, \frac{\partial}{\partial \textbf{y}^b}]=0,
\end{eqnarray}
where
\begin{eqnarray}
&&(i)\ \ \textbf{R}^c_{\ ab}=\frac{\D \textbf{G}^c_a}{\D \textbf{x}^b}-\frac{\D \textbf{G}^c_b}{\D \textbf{x}^a},\\
&&(ii)\ \ \textbf{G}^c_{\ ab}=\frac{\partial \textbf{G}^c_a}{\partial \textbf{y}^b}.\label{p14}
\end{eqnarray}

\bigskip
Therefore, we have the following.
\begin{cor}
Let $(M_1\times{}_{f}M_2,F)$ be a twisted product
Finsler manifold.  Then
\[
\textbf{R}^c_{\ ab}=(\textbf{R}^k_{\ ij}, \textbf{R}^k_{\ i\beta}, \textbf{R}^k_{\ \alpha j}, \textbf{R}^k_{\ \alpha \beta}, \textbf{R}^\gamma_{\ ij}, \textbf{R}^\gamma_{\ i\beta}, \textbf{R}^\gamma_{\ \alpha j}, \textbf{R}^\gamma_{\ \alpha\beta})
\]
where
\[
\textbf{R}^k_{\ ij}:=\frac{\D \textbf{G}^k_i}{\D x^j}-\frac{\D \textbf{G}^k_j}{\D x^i},\ \ \ \textbf{R}^k_{\ i\beta}:=\frac{\D \textbf{G}^k_i}{\D u^\beta}-\frac{\D \textbf{G}^k_\beta}{\D x^i},
\]
\[
\textbf{R}^k_{\ \alpha j}:=\frac{\D \textbf{G}^k_\alpha}{\D x^j}-\frac{\D \textbf{G}^k_j}{\D u^\alpha},\ \ \ \textbf{R}^k_{\ \alpha\beta}:=\frac{\D \textbf{G}^k_\alpha}{\D u^\beta}-\frac{\D \textbf{G}^k_\beta}{\D u^\alpha},
\]
\[
\textbf{R}^\gamma_{\ ij}:=\frac{\D \textbf{G}^\gamma_i}{\D x^j}-\frac{\D \textbf{G}^\gamma_j}{\D x^i},\ \ \ \textbf{R}^\gamma_{\ i\beta}:=\frac{\D \textbf{G}^\gamma_i}{\D u^\beta}-\frac{\D \textbf{G}^\gamma_\beta}{\D x^i},
\]
\[
\textbf{R}^\gamma_{\ \alpha j}:=\frac{\D \textbf{G}^\gamma_\alpha}{\D x^j}-\frac{\D \textbf{G}^\gamma_j}{\D u^\alpha},\ \ \ \textbf{R}^\gamma_{\ \alpha\beta}:=\frac{\D \textbf{G}^\gamma_\alpha}{\D u^\beta}-\frac{\D \textbf{G}^\gamma_\beta}{\D u^\alpha},
\]
\end{cor}

\bigskip
With a simple calculation, we have the following.

\begin{cor}
Let $(M_1\times{}_{f}M_2,F)$ be a twisted product
Finsler manifold.  Then
\[
\textbf{G}^c_{ab}=(\textbf{G}^k_{ij}, \textbf{G}^k_{i\beta}, \textbf{G}^k_{\alpha j}, \textbf{G}^k_{\alpha \beta}, \textbf{G}^\gamma_{ij}, \textbf{G}^\gamma_{i\beta}, \textbf{G}^\gamma_{\alpha j}, \textbf{G}^\gamma_{\alpha\beta})
\]
where
\begin{eqnarray*}
\textbf{G}^\gamma_{\alpha\beta}\!\!\!\!&=&\!\!\!\!\frac{\partial\textbf{G}^\gamma_\alpha}{\partial
v^\beta}=G^\gamma_{\alpha\beta}+f^{-1}(C^{\gamma\lambda}_{\alpha;\beta}f_\lambda F_2^2+2C^{\gamma\lambda}_\alpha f_\lambda v_\beta+2C^{\gamma\lambda}_\beta f_\lambda v_\alpha\\
\!\!\!\!&&\!\!\!\!\hspace{1.2cm}-f^\gamma g_{\alpha\beta}+f_\beta\delta^\gamma_\alpha+f_\alpha\delta^\gamma_\beta)=\textbf{G}^\gamma_{\beta\alpha},\\
\textbf{G}^k_{ij}\!\!\!\!&=&\!\!\!\!\frac{\partial\textbf{G}^k_i}{\partial
y^j}=G^k_{ij}+C^{kh}_{i;j}ff_hF^2_2
=\textbf{G}^k_{ji},\\
\textbf{G}^k_{i\beta}\!\!\!\!&=&\!\!\!\!\frac{\partial\textbf{G}^k_i}{\partial v^\beta}=2C^{kh}_iff_hv_\beta
=\textbf{G}^k_{\beta i},\\
\textbf{G}^k_{\alpha\beta}\!\!\!\!&=&\!\!\!\!\frac{\partial\textbf{G}^k_\alpha}{\partial v^\beta}=-ff^kg_{\alpha\beta}=\textbf{G}^k_{\beta\alpha},\\
\textbf{G}^\gamma_{i\beta}\!\!\!\!&=&\!\!\!\!\frac{\partial\textbf{G}^\gamma_i}{\partial v^\beta}=f^{-1}f_i\delta^\gamma_\beta=\textbf{G}^\gamma_{\beta i},\\
\textbf{G}^\gamma_{ij}\!\!\!\!&=&\!\!\!\!\frac{\partial
\textbf{G}_i^\gamma}{\partial y^j}=\textbf{G}^\gamma_{ji}=0.
\end{eqnarray*}
where $C^{kh}_{i;j}=\frac{\partial C^{kh}_i}{\partial y^j}$. Apart from $\textbf{G}^c_{ab}$, the functions $\textbf{F}^c_{ab}$ given by
\begin{equation}
\textbf{F}^c_{ab}=\frac{1}{2}\textbf{g}^{ce}\Big(\frac{\D
\textbf{g}_{ea}}{\D \textbf{x}^b}+\frac{\D \textbf{g}_{eb}}{\D
\textbf{x}^a}-\frac{\D \textbf{g}_{ab}}{\D
\textbf{x}^e}\Big)\label{hor}
\end{equation}
\end{cor}

\bigskip

\begin{cor}\label{cor}
Let $(M_1\times{}_{f}M_2,F)$ be a twisted product
Finsler manifold.  Then
\[
\textbf{F}^c_{ab}=(\textbf{F}^k_{ij}, \textbf{F}^k_{i\beta}, \textbf{F}^k_{\alpha j}, \textbf{F}^k_{\alpha \beta}, \textbf{F}^\gamma_{ij}, \textbf{F}^\gamma_{i\beta}, \textbf{F}^\gamma_{\alpha j}, \textbf{F}^\gamma_{\alpha\beta})
\]
where
\begin{eqnarray}
\textbf{F}^k_{ij}\!\!\!\!&=&\!\!\!\!F^k_{ij}-\Big(M^r_jC^k_{ir}
+M^r_iC^k_{jr}
-M^r_hC_{ijr}g^{kh}\Big),\label{hor2}\\
\textbf{F}^k_{i\beta}\!\!\!\!&=&\!\!\!\!
-\textbf{G}^r_\beta C^k_{ir}=\textbf{F}^k_{\beta i},\label{hor3}\\
\textbf{F}^k_{\alpha\beta}\!\!\!\!&=&\!\!\!\!-ff^k
g_{\alpha\beta}+f^2g^{kh}\textbf{G}^\lambda_h C_{\alpha\beta\lambda},\label{hor4}
\\
\textbf{F}^\gamma_{ij}\!\!\!\!&=&\!\!\!\!f^{-2}g^{\gamma\lambda}
\textbf{G}^r_\lambda C_{ijr},\label{hor5}\\
\textbf{F}^\gamma_{i\beta}\!\!\!\!&=&\!\!\!\!f^{-1}f_i\delta^\gamma_\beta-\textbf{G}^\alpha_iC^\gamma_{\alpha\beta}=
\textbf{F}^\gamma_{\beta i},\label{hor6}\\
\textbf{F}^\gamma_{\alpha\beta}\!\!\!\!&=&\!\!\!\!F^\gamma_{\alpha\beta}+N^\gamma_{\alpha\beta}-\Big(M^\mu_\beta C^\gamma_{\alpha\mu}
 +M^\mu_\alpha C^\gamma_{\beta\mu}
-M^\mu_\lambda C_{\alpha\beta\mu}g^{\gamma\lambda}
\Big),\label{hor7}
\\
\nonumber F^k_{ij}\!\!\!\!&=&\!\!\!\!\frac{1}{2}g^{kh}(\frac{\delta g_{hi}}{\delta x^j}+\frac{\delta g_{hj}}{\delta x^i}-\frac{\delta g_{ij}}{\delta x^h}),\\
\nonumber  F^\gamma_{\alpha\beta}\!\!\!\!&=&\!\!\!\!\frac{1}{2}g^{\gamma\lambda}(\frac{\delta g_{\lambda\alpha}}{\delta u^\beta}+\frac{\delta g_{\lambda\beta}}{\delta u^\alpha}-\frac{\delta g_{\alpha\beta}}{\delta u^\lambda}),\\
\nonumber M^r_i\!\!\!\!&=&\!\!\!\!C^{rh}_iff_hF^2_2,\\
\nonumber  M^\mu_\alpha\!\!\!\!&=&\!\!\!\!f^{-1}(C^{\mu\gamma}_\alpha f_\gamma F^2_2+f_ry^r\delta^\mu_\alpha+f_\gamma v^\gamma \delta^\mu_\alpha-g^{\mu\gamma}f_\gamma v_\alpha+f_\alpha v^\mu), \\
\nonumber  N^\gamma_{\alpha\beta}\!\!\!\!&=&\!\!\!\!f^{-1}(f_\beta \delta^\gamma_\alpha+f_\alpha \delta^\gamma_\beta-f_\lambda g^{\gamma\lambda}g_{\alpha\beta}).
\end{eqnarray}
\end{cor}
\begin{proof}
By using (\ref{hor}) we have
\begin{equation}
\textbf{F}^k_{ij}=\frac{1}{2}g^{kh}\Big(\frac{\D g_{hi}}{\D
x^j}+\frac{\D g_{hj}}{\D x^i}-\frac{\D g_{ij}}{\D
x^h}\Big).\label{hor1}
\end{equation}
Since $g_{ij}$ is a function with respect to $(x,y)$, then by (\ref{spray7}) and (\ref{dec1}) we obtain
\[
\frac{\D g_{hi}}{\D x^j}=
\frac{\delta g_{hi}}{\delta x^j}-2M^r_jC_{hir}.
\]
Interchanging $i$, $j$ and $h$ in the above equation gives us
\begin{eqnarray*}
\frac{\D g_{hj}}{\D x^i}=\frac{\delta g_{hj}}{\delta x^i}-2M^r_iC_{hjr}\\
\frac{\D g_{ij}}{\D x^h}= \frac{\delta g_{ij}}{\delta x^h}-2M^r_hC_{ijr}.
\end{eqnarray*}
Putting these equation in (\ref{hor1}), give us (\ref{hor2}). In the similar
way, we can prove the another relation.
\end{proof}
By using (i) of (\ref{spray6}) and (\ref{hor2})-(\ref{hor7}),  we can conclude the following.
\begin{lem}\label{lem}
Let $(M_1\times{}_{f}M_2,F)$ be a twisted product
Finsler manifold. Then  $\textbf{y}^c\textbf{F}^a_{bc}=\textbf{G}^a_b$, where
$\textbf{F}^a_{bc}$ and $\textbf{G}^a_{b}$ are defined by
(\ref{hor}) and (i) of (\ref{spray6}), respectively.
\end{lem}

\smallskip

The Cartan torsion is one of the most important non-Riemannian quantity in Finsler geometry and it first introduced by Finsler  and emphased by Cartan which measures a departure from a Riemannian manifold. More precisely,  a Finsler metric reduces to a Riemannian metric if and only if it has vanishing Cartan torsion. The local components of  Cartan tensor field of the twisted Finsler manifold $(M_1\times{}_{f}M_2,F)$ is defined by
\[
\textbf{C}^a_{bc}=\frac{1}{2}\textbf{g}^{ae}\frac{\partial \textbf{g}_{be}}{\partial \textbf{y}^c}.
\]
From this definition, we conclude the following.
\begin{lem}\label{lemC}
 Let $C^k_{ij}$ and $C^\gamma_{\alpha\beta}$ be the local
components of Cartan tensor field on $M_1$ and $M_2$,
respectively. Then we have
\[
\textbf{C}^c_{ab}=(\textbf{C}^k_{ij}, \textbf{C}^k_{i\beta},
\textbf{C}^k_{\alpha j}, \textbf{C}^k_{\alpha \beta},
\textbf{C}^\gamma_{ij}, \textbf{C}^\gamma_{i\beta},
\textbf{C}^\gamma_{\alpha j}, \textbf{C}^\gamma_{\alpha\beta}),
\]
where
\begin{eqnarray*}
&&\textbf{C}^k_{ij}=\frac{1}{2}g^{kh}\frac{\partial g_{ij}}{\partial
y^h}=C^k_{ij},\\
&&\textbf{C}^\gamma_{\alpha\beta}=\frac{1}{2}g^{\gamma\lambda}\frac{\partial
g_{\alpha\beta}}{\partial v^\lambda}=C^\gamma_{\alpha\beta},
\end{eqnarray*}
and $\textbf{C}^k_{i\beta}=\textbf{C}^k_{\alpha j}=\textbf{C}^k_{\alpha \beta}=\textbf{C}^\gamma_{ij}=\textbf{C}^\gamma_{i\beta}=\textbf{C}^\gamma_{\alpha j}=0$.
\end{lem}
By using the Lemma \ref{lemC}, we  can get the following.
\begin{cor}\label{1}
Let $(M_1\times{}_{f}M_2,F)$ be a twisted product
Finsler manifold. Then $(M_1\times{}_{f}M_2,F)$ is a
Riemannian manifold if and only if $(M_1, F_1)$ and $(M_2, F_2)$
are Riemannian manifold.
\end{cor}

\bigskip

Various interesting special forms of Cartan tensors have been obtained by some Finslerians \cite{Mat4}. The Finsler spaces having such special forms have been called C-reducible, C2-like, semi-C-reducible, and etc. In \cite{M3}, Matsumoto  introduced the notion of C-reducible Finsler metrics and proved that any Randers metric is C-reducible. Later on, Matsumoto-H\={o}j\={o} proves that the converse is true too \cite{MH}.

Here, we define the Matsumoto twisted tensor $\textbf{M}_{abc}$ for a twisted product
Finsler manifold $(M_1\times{}_{f}M_2,F)$ as follows:
\[
\textbf{M}_{abc}=\textbf{C}_{abc}-\frac{1}{n+1}\{\textbf{I}_a\textbf{h}_{bc}+\textbf{I}_b\textbf{h}_{ac}
+\textbf{I}_c\textbf{h}_{ab}\},
\]
where $\textbf{I}_a=\textbf{g}^{bc}\textbf{C}_{abc}$,
$\textbf{C}_{abc}=\textbf{g}_{cd}\textbf{C}^d_{ab}$ and
$\textbf{h}_{ab}=\textbf{g}_{ab}-\frac{1}{F^2}\textbf{y}_a\textbf{y}_b$.
By attention to the above equation and relations
\[
\textbf{C}_{ijk}=C_{ijk}, \ \  \ \  \textbf{C}_{\alpha\beta\gamma}=f^2C_{\alpha\beta\gamma},
\]
we obtain
\[
\textbf{M}_{\alpha
jk}=-\frac{1}{n+1}\Big\{I_\alpha(g_{jk}-\frac{1}{F^2}y_jy_k)-\frac{f^2}{F^2}v_\alpha(I_j
y_k+I_ky_j)\Big\}.
\]
Contracting the above equation in $y^jy^k$ give us
\[
y^jy^k\textbf{M}_{\alpha
jk}=-\frac{f^2F_1^2F_2^2}{(n+1)F^2}I_\alpha.
\]
Similarly, we obtain
\[
v^\lambda v^\beta\textbf{M}_{i\beta\lambda}=-\frac{f^2F_1^2F_2^2}{(n+1)F^2}I_i.
\]
Therefore if $\textbf{M}_{i\beta\lambda}=\textbf{M}_{\alpha jk}=0$,  then we get $I_i=I_\alpha=0$, i.e.,
$(M_1, F_1)$ and $(M_2, F_2)$ are Riemannian manifolds. Thus we have
\begin{thm}
There is not exist any  C-reducible twisted product Finsler manifold.
\end{thm}

Now, we are going to consider semi-C-reducible twisted product
Finsler manifold $(M_1\times{}_{f}M_2,F)$. Let $(M_1\times{}_fM_2,F)$ be a semi-C-reducible twisted product Finsler manifold. Then we have
\[
\textbf{C}_{abc}=\frac{p}{n+1}\{\textbf{I}_a\textbf{h}_{bc}+\textbf{I}_b\textbf{h}_{ac}
+\textbf{I}_c\textbf{h}_{ab}\}+\frac{q}{\textbf{C}^2}\textbf{I}_a\textbf{I}_b\textbf{I}_c,
\]
where $\textbf{C}^2=\textbf{I}^a\textbf{I}_a$ and $p$ and $q$ are scalar function on $M_1\times{}_fM_2$ with $p+q=1$. This equation gives us
\[
0=\textbf{C}_{\alpha jk}=\frac{p}{n+1}\Big\{I_\alpha(g_{jk}-\frac{1}{F^2}y_jy_k)-\frac{f^2}{F^2}v_\alpha(I_j
y_k+I_ky_j)\Big\}+\frac{q}{\textbf{C}^2}I_\alpha I_jI_k.
\]
Contraction the above equation with $y^jy^k$ implies that
\[
pf^2F_1^2F_2^2I_\alpha=0.
\]
Therefore we have $p=0$ or $I_\alpha=0$. If $p=0$, then $F$ is $C2$-like metric. But if $p\neq 0$, then $I_\alpha=0$, i. e., $F_2$ is Riemannian metric. In this case, with similar way we conclude that $F_1$ is Riemannian metric. But, by definition $F$ can not be a Riemannian metric. Therefore we have
\begin{thm}
Every semi-C-reducible twisted product Finsler manifold $(M_1\times{}_fM_2,F)$ is a $C2$-like manifold.
\end{thm}
\section{Riemannian Curvature}
The Riemannian curvature of twisted product Finsler manifold
$(M_1\times{}_{f}M_2,F)$ with respect to Berwald
connection is given by
\begin{equation}
\textbf{R}^{\ a}_{b\ cd}=\frac{\D \textbf{F}^a_{bc}}{\D \textbf{x}^d}-\frac{\D \textbf{F}^a_{bd}}{\D \textbf{x}^c}+\textbf{F}^a_{de}\textbf{F}^e_{bc}
-\textbf{F}^a_{ce}\textbf{F}^e_{bd}.\label{cur}
\end{equation}
\begin{lem}
Let $(M_1\times_fM_2, F)$ be a twisted product Finsler manifold.  Then we have
\[
\textbf{R}^a_{\ cd}=\textbf{y}^b\textbf{R}^{\ a}_{b\ cd},
\]
where $\textbf{R}^a_{\ cd}$ and $\textbf{y}^b\textbf{R}^{\ a}_{b\ cd}$ are given by (\ref{p14}) and (\ref{cur}).
\end{lem}
\begin{proof}
By using (\ref{cur}), we have
\begin{equation}
\textbf{y}^b\textbf{R}^{\ i}_{b\ kl}=\textbf{y}^b\frac{\D \textbf{F}^i_{bk}}{\D \textbf{x}^l}-\textbf{y}^b\frac{\D \textbf{F}^i_{bl}}{\D \textbf{x}^k}+\textbf{y}^b\textbf{F}^i_{le}\textbf{F}^e_{bk}.\label{p11}
-\textbf{y}^b\textbf{F}^i_{ke}\textbf{F}^e_{bl}.
\end{equation}
By using Corollary \ref{cor} and Lemma \ref{lem}, we obtain
\begin{eqnarray}
&&\textbf{y}^b\frac{\D \textbf{F}^i_{bk}}{\D \textbf{x}^l}
=\frac{\D \textbf{G}^i_k}{\D x^l}+\textbf{F}^i_{jk}\textbf{G}^j_l+\textbf{F}^i_{\beta k}\textbf{G}^\beta_l,\label{p112}\\
&& \textbf{y}^b\textbf{F}^i_{le}\textbf{F}^e_{bk}=\textbf{F}^i_{lh}\textbf{G}^h_k+\textbf{F}^i_{l\gamma}\textbf{G}_k^\gamma.
\label{p12}
\end{eqnarray}
Interchanging $i$ and $j$ in the above equation imply that
\begin{eqnarray}
&&\textbf{y}^b\frac{\D \textbf{F}^i_{bl}}{\D \textbf{x}^k}
=\frac{\D \textbf{G}^i_l}{\D x^k}+\textbf{F}^i_{jl}\textbf{G}^j_k+\textbf{F}^i_{\beta l}\textbf{G}^\beta_k,\label{p113}\\
&&\textbf{y}^b\textbf{F}^i_{ke}\textbf{F}^e_{bl}=\textbf{F}^i_{kh}\textbf{G}^h_l+\textbf{F}^i_{k\gamma}\textbf{G}_l^\gamma.
\label{p13}
\end{eqnarray}
Setting (\ref{p112}), (\ref{p12}), (\ref{p113}) and (\ref{p13}) in (\ref{p11}) give us $\textbf{y}^b\textbf{R}^{\ i}_{b\ kl}=\textbf{R}^i_{\ kl}$. In the similar way, we can obtain this relation for another indices.
\end{proof}

\bigskip

Using (\ref{cur}), we can compute the Riemannian curvature of a twisted product Finsler manifold.
\begin{lem}\label{thm5}
Let $(M_1\times_fM_2, F)$ be a twisted product Finsler manifold . Then the coefficients of Riemannian curvature  are as follows:
\begin{eqnarray}
\textbf{R}^{\ i}_{j\ kl}\!\!\!\!&=&\!\!\!\!R^{\ i}_{j\ kl}-\Big\{\{M^r_l\frac{\partial F^i_{jk}}{\partial y^r}
+\frac{\D M^i_{jk}}{\D x^l}+F^i_{lh}M^h_{jk}+M^i_{lh}F^h_{jk}-M^i_{lh}M^h_{jk}\nonumber\\
\!\!\!\!&&\!\!\!\!+f^{-2}g^{\alpha\gamma}
\textbf{G}^r_\alpha\textbf{G}^m_\gamma C^i_{lr}C_{jkm}\}
-\mathfrak{C}^k_l\Big\}\label{cur0}
\\
\textbf{R}^{\ i}_{\alpha\
kl}\!\!\!\!&=&\!\!\!\!\Big\{-\frac{\D}{\D
x^l}(\textbf{G}^r_\alpha C^i_{kr}
)-(F^i_{rl}-M^i_{rl})\textbf{G}^m_\alpha
C^r_{km}-f^{-1}\textbf{G}^r_\beta C^i_{lr}f_k\delta^\beta_\alpha\nonumber\\
\!\!\!\!&&\!\!\!\!+
\textbf{G}^r_\beta\textbf{G}^\mu_k C^i_{lr}C^\beta_{\alpha\mu}
\Big\}-\mathfrak{C}^k_l.
\end{eqnarray}
\begin{eqnarray}
\textbf{R}^{\ i}_{j\
\beta\lambda}\!\!\!\!&=&\!\!\!\!\Big\{-\frac{\D}{\D
u^\lambda}(\textbf{G}^r_\beta C^i_{jr} )+\textbf{G}^m_\lambda
\textbf{G}^l_\beta C^i_{rm}C^r_{jl}-(f^i
g_{\alpha\lambda}-f\textbf{G}^\mu_hg^{ih} C_{\alpha\lambda\mu}
)\nonumber\\
\!\!\!\!&&\!\!\!\!(f_j\delta^\alpha_\beta-f \textbf{G}^\nu_jC^\alpha_{\beta\nu}
)\Big\}-\mathfrak{C}^\beta_\lambda.
\end{eqnarray}
\begin{eqnarray}
\textbf{R}^{\ i}_{\alpha\ \beta l}\!\!\!\!&=&\!\!\!\!\frac{\D}{\D
u^\beta}( \textbf{G}^r_\alpha C^i_{lr}
)-\frac{\D}{\D
x^l}f(f^i
g_{\alpha\beta}-f\textbf{G}^\lambda_h
g^{ih}C_{\alpha\beta\lambda})\nonumber\\
\!\!\!\!&&\!\!\!\! -
\textbf{G}^m_\beta\textbf{G}^s_\alpha C^i_{rm}C^r_{ls}
+(f^i
g_{\mu\beta}-f\textbf{G}^\lambda_h
g^{ih}C_{\mu\beta\lambda}
)(f_l
\delta^\mu_\alpha\nonumber\\
\!\!\!\!&&\!\!\!\!-f\textbf{G}^\nu_l
C^\mu_{\alpha\nu})-fg^{rh}(F^i_{rl}-M^i_{rl})(f_h
g_{\alpha\beta}-f\textbf{G}^\lambda_hC_{\alpha\beta\lambda}
)\nonumber\\
\!\!\!\!&&\!\!\!\!-\textbf{G}^r_\mu
C^i_{lr}(F^\mu_{\alpha\beta}+N^\mu_{\alpha\beta}-M^\mu_{\alpha\beta}).
\end{eqnarray}
\begin{eqnarray}
\textbf{R}^{\ i}_{j\ \beta l}\!\!\!\!&=&\!\!\!\!-\frac{\D}{\D
x^l}( \textbf{G}^r_\beta C^i_{jr} )-\frac{\D}{\D
u^\beta}(F^i_{jl}-M^i_{jl})-(F^i_{lr}-M^i_{lr})\textbf{G}^s_\beta C^r_{js}\nonumber\\
\!\!\!\!&&\!\!\!\!-f^{-1}\textbf{G}^r_\alpha C^i_{lr}(f_j\delta^\alpha_\beta-f\textbf{G}^\mu_jC^\alpha_{\beta\mu})+
\textbf{G}^s_\beta C^i_{rs}(F^r_{jl}-M^r_{jl})\nonumber\\
\!\!\!\!&&\!\!\!\!+f^{-1}\textbf{G}^r_\mu
C_{jlr}(f^i\delta^\mu_\beta-f\textbf{G}^\lambda_hg^{ih}C^\mu_{\beta\lambda})
\end{eqnarray}
\begin{eqnarray}
\textbf{R}^{\ i}_{\alpha\ \beta \lambda}\!\!\!\!&=&\!\!\!\!\Big\{-\frac{\D}{\D
u^\lambda}( ff^ig_{\alpha\beta}-f^2g^{ih}\textbf{G}^\mu_h C_{\alpha\beta\mu}
)+f\textbf{G}^s_\lambda C^i_{rs}(f^rg_{\alpha\beta}\nonumber\\
\!\!\!\!&&\!\!\!\!-f\textbf{G}^\mu_l C_{\alpha\beta\mu}g^{rl})
-f(f^ig_{\lambda\mu}-fg^{ih}\textbf{G}^\kappa_hC_{\lambda\mu\kappa})(F^\mu_{\alpha\beta}\nonumber\\
\!\!\!\!&&\!\!\!\!+N^\mu_{\alpha\beta}-M^\mu_{\alpha\beta})\Big\}-\mathfrak{C}^\beta_\lambda.
\end{eqnarray}
\begin{eqnarray}
\textbf{R}^{\ \gamma}_{j\
kl}\!\!\!\!&=&\!\!\!\!\Big\{\frac{\D}{\D
x^l}(f^{-2}
g^{\gamma\lambda}\textbf{G}^r_\lambda C_{jkr}
)+f^{-2}g^{\gamma\lambda}\textbf{G}^s_\lambda C_{lrs}(F^r_{jk}-M^r_{jk})
\nonumber\\
\!\!\!\!&&\!\!\!\!+f^{-3}\textbf{G}^r_\mu C_{jkr}(f_l
g^{\gamma\mu}-f\textbf{G}^\alpha_l C^{\gamma\mu}_{\alpha})
\Big\}-\mathfrak{C}^k_l.
\end{eqnarray}
\begin{eqnarray}
\textbf{R}^{\ \gamma}_{j\
\beta l}\!\!\!\!&=&\!\!\!\!\frac{\D}{\D
x^l}(f^{-1}f_j
\delta^\gamma_\beta-f\textbf{G}^\alpha_j
C^\gamma_{\alpha\beta})-\frac{\D}{\D u^\beta}(f^{-2}g^{\gamma\lambda}\textbf{G}^r_\lambda C_{jlr})\nonumber\\
\!\!\!\!&&\!\!\!\!-f^{-2}g^{\gamma\lambda}\textbf{G}^s_\lambda\textbf{G}^m_\beta
C^h_{ls}C_{hjm}+f^{-2}(f_l\delta^\gamma_\mu-f\textbf{G}^\alpha_lC^\gamma_{\mu\alpha})
(f_j\delta^\mu_\beta\nonumber\\
\!\!\!\!&&\!\!\!\!-f\textbf{G}^\nu_jC^\mu_{\beta\nu})
-f^{-1}(f_r
\delta^\gamma_\beta-f\textbf{G}^\alpha_r
C^\gamma_{\beta\alpha})(F^r_{jl}-M^r_{jl})\nonumber\\
\!\!\!\!&&\!\!\!\!
-f^{-2}g^{\mu\lambda}\textbf{G}^r_\lambda C_{jlr}(F^\gamma_{\beta\mu}+N^\gamma_{\beta\mu}-M^\gamma_{\beta\mu}).
\end{eqnarray}
\begin{eqnarray}
\textbf{R}^{\ \gamma}_{\alpha\
\beta l}\!\!\!\!&=&\!\!\!\!\frac{\D}{\D
x^l}(F^\gamma_{\alpha\beta}+N^\gamma_{\alpha\beta}-M^\gamma_{\alpha\beta}
)-\frac{\D}{\D u^\beta}(f^{-1}f_l\delta^\gamma_\alpha-\textbf{G}^\mu_l C^\gamma_{\alpha\mu})\nonumber\\
\!\!\!\!&&\!\!\!\!-f^{-1}g^{\gamma\lambda}\textbf{G}^s_\lambda C^h_{ls}(f_hg_{\alpha\beta}-f\textbf{G}^\mu_hC_{\alpha\beta\mu}) +f^{-1}(f_l\delta^\gamma_\mu\nonumber\\
\!\!\!\!&&\!\!\!\!-f\textbf{G}^\kappa_lC^\gamma_{\mu\kappa})
(F^\mu_{\alpha\beta}+N^\mu_{\alpha\beta}-M^\mu_{\alpha\beta})+f^{-1}\textbf{G}^s_\alpha C^r_{ls}(f_r\delta^\gamma_\beta
\nonumber\\
\!\!\!\!&&\!\!\!\!-f\textbf{G}^\kappa_rC^\gamma_{\beta\kappa})
-f^{-1}(f_l
\delta^\mu_\alpha
-f\textbf{G}^\kappa_l
C^\mu_{\alpha\kappa})
(F^\gamma_{\beta\mu}+N^\gamma_{\beta\mu}-M^\gamma_{\beta\mu}).
\end{eqnarray}
\begin{eqnarray}
\textbf{R}^{\ \gamma}_{j\
\beta\lambda}\!\!\!\!&=&\!\!\!\!\Big\{\frac{\D}{\D
u^\lambda}(f^{-1}f_j
\delta^\gamma_\beta-\textbf{G}^\alpha_j
C^\gamma_{\alpha\beta})+f^{-1}(F^\gamma_{\alpha\lambda}+N^\gamma_{\alpha\lambda}-M^\gamma_{\alpha\lambda})
\nonumber\\
\!\!\!\!&&\!\!\!\!(f_j
\delta^\alpha_\beta-f\textbf{G}^\nu_j
C^\alpha_{\beta\nu})-f^{-1}\textbf{G}^m_\beta C^r_{jm}
(f_r
\delta^\gamma_\lambda
-f\textbf{G}^\alpha_r
C^\gamma_{\lambda\alpha})\Big\}-\mathfrak{C}^\beta_\lambda.
\\
\textbf{R}^{\ \gamma}_{\alpha\
kl}\!\!\!\!&=&\!\!\!\!\Big\{\frac{\D}{\D
x^l}(f^{-1}f_k
\delta^\gamma_\alpha-\textbf{G}^\mu_k C^\gamma_{\alpha\mu}
)+f^{-2}(f_l
\delta^\gamma_\beta-f\textbf{G}^\kappa_lC^\gamma_{\beta\kappa} )\nonumber\\
\!\!\!\!&&\!\!\!\! (f_k
\delta^\beta_\alpha-f\textbf{G}^\nu_k
C^\beta_{\alpha\nu})-f^{-2}\textbf{G}^s_\mu \textbf{G}^m_\alpha
g^{\gamma\mu}C^h_{ls}C_{hkm}
\Big\}-\mathfrak{C}^k_l.
\end{eqnarray}
\begin{eqnarray}
\textbf{R}^{\ \gamma}_{\alpha\
\beta\lambda}\!\!\!\!&=&\!\!\!\!R^{\ \gamma}_{\alpha\
\beta\lambda}-\Big\{\{M^\kappa_\lambda\frac{\partial
F^\gamma_{\alpha\beta}}{\partial v^\kappa} +\frac{\D
M^\gamma_{\alpha\beta}}{\D
u^\lambda}+F^\gamma_{\lambda\mu}M^\mu_{\alpha\beta}+M^\gamma_{\lambda\mu}F^\mu_{\alpha\beta}
-M^\gamma_{\lambda\mu}M^\mu_{\alpha\beta}\nonumber\\
\!\!\!\!&&\!\!\!\!+\frac{\D N^\gamma_{\alpha\lambda}}{\D
u^\beta}+F^\gamma_{\beta\mu}N^\mu_{\alpha\lambda}+N^\gamma_{\beta\mu}F^\mu_{\alpha\lambda}
+N^\gamma_{\beta\mu}N^\mu_{\alpha\lambda}+N^\gamma_{\lambda\mu}M^\mu_{\alpha\beta}
+N^\mu_{\alpha\beta}M^\gamma_{\lambda\mu}
\nonumber\\
\!\!\!\!&&\!\!\!\!+(g^{rs}\delta^\gamma_\lambda f_s
 -fg^{rs}\textbf{G}^\kappa_sC^\gamma_{\lambda\kappa}
)(g_{\alpha\beta}f_r
-f\textbf{G}^\mu_rC_{\alpha\beta\mu})\}
-\mathfrak{C}^\beta_\lambda\Big\}\label{cur1}
\end{eqnarray}
where
\begin{eqnarray*}
M^i_{jk}\!\!\!\!&=&\!\!\!\! M^r_kC^i_{jr}+M^r_jC^i_{kr}-M^r_hg^{ih}C_{jkr},\\
M^\gamma_{\alpha\beta}\!\!\!\!&=&\!\!\!\!M^\mu_\beta C^\gamma_{\alpha\mu}+M^\mu_\alpha C^\gamma_{\beta\mu}-M^\mu_\nu g^{\gamma\nu}C_{\alpha\beta\mu},\\
N^\gamma_{\alpha\beta}\!\!\!\!&=&\!\!\!\!f^{-1}(f_\beta\delta^\gamma_\alpha+f_\alpha\delta^\gamma_\beta-f_\lambda g^{\gamma\lambda}g_{\alpha\beta})
\end{eqnarray*}
and $\mathfrak{C}^i_j$ denotes the interchange of indices $i$, $j$ and subtraction.
\end{lem}

\bigskip

By the Theorem \ref{thm5}, we have the following.
\begin{thm}
Let $(M_1\times{}_{f}M_2,F)$ be a flat twisted
product Finsler manifold and $(M_1, F_1)$ is Riemannian. If $f$ is a function on $M_2$, only, then $(M_1, F_1)$ is locally flat.
\end{thm}

\smallskip

Similarly, we get the following.

\begin{thm}
Let $(M_1\times{}_{f}M_2,F)$ be a flat twisted
product Finsler manifold and $(M_2, F_2)$ is Riemannian. If $f$ is a function on $M_1$, only, then $(M_2, F_2)$ is a space of positive constant curvature $||grad f||^2$.
\end{thm}
\begin{proof}
Since $M_2$ is Riemannain and $f$ is a function on $M_1$, then by (\ref{cur1}), we obtain
\begin{equation}
\textbf{R}^{\ \gamma}_{\alpha\
\beta\lambda}=R^{\ \gamma}_{\alpha\
\beta\lambda}+||grad f||^2(\delta^\gamma_\lambda g_{\alpha\beta}-\delta^\gamma_\beta g_{\alpha\lambda}).
\end{equation}
Since $(M_1\times{}_{f}M_2,F)$ is flat, then  $\textbf{R}^{\ \gamma}_{\alpha\
\beta\lambda}=0$. Thus the proof is complete.
\end{proof}

\bigskip

\begin{thm}
Let $(M_1\times{}_{f}M_2,F)$ be a twisted product Riemannian manifold and $f$ be a function on $M_2$, only. Then $(M_1\times{}_{f}M_2,F)$ is flat if and only if $(M_1, F_1)$ is flat and the Riemannian curvature of $(M_2, F_2)$ satisfies in the following equation:
\begin{equation}
R^{\ \gamma}_{\alpha\
\beta\lambda}=\{\frac{\D N^\gamma_{\alpha\lambda}}{\D
u^\beta}+F^\gamma_{\beta\mu}N^\mu_{\alpha\lambda}+N^\gamma_{\beta\mu}F^\mu_{\alpha\lambda}
+N^\gamma_{\beta\mu}N^\mu_{\alpha\lambda}\}
-\mathfrak{C}^\beta_\lambda.
\end{equation}
\end{thm}
\section{Twisted Product Finsler manifolds With Non-Riemannian Curvature Properties}
There are several important non-Riemannian  quantities such as the  Berwald curvature ${\bf B}$, the mean Berwald curvature ${\bf E}$ and the Landsberg curvature ${\bf L}$, etc \cite{TP1}. They all vanish  for Riemannian metrics, hence they are said to be   non-Riemannian. In this section, we find some necessary and sufficient conditions under which a  twisted product Riemannian manifold are Berwaldian, of isotropic Berwald curvature, of isotropic mean Berwald curvature. First, we prove the following.
\begin{lem}\label{Berwald}
Let $(M_1\times_fM_2, F)$ be a twisted product Finsler manifold . Then the coefficients of Berwald curvature  are as follows:
\begin{eqnarray}
\textbf{B}^\gamma_{\alpha\beta\lambda}\!\!\!\!&=&\!\!\!\!B^\gamma_{\alpha\beta\lambda}
+f^{-1}(C^{\gamma\nu}_{\lambda;\alpha;\beta}f_\nu F_2^2+2C^{\gamma\nu}_{\alpha;\beta}f_\nu v_\lambda+2C^{\gamma\nu}_{\alpha;\lambda}f_\nu v_\beta\nonumber\\
\!\!\!\!&&\!\!\!\!+2C^{\gamma\nu}_\alpha f_\nu g_{\lambda\beta}+2C^{\gamma\nu}_{\lambda;\beta} f_\nu v_\alpha+2C^{\gamma\nu}_\beta f_\nu g_{\lambda\alpha}\nonumber\\
\!\!\!\!&&\!\!\!\!+2C^{\gamma\nu}_\lambda f_\nu g_{\alpha\beta}-2C_{\alpha\beta\lambda}f^\gamma),\\
\textbf{B}^k_{ijl}\!\!\!\!&=&\!\!\!\!B^k_{ijl}+fC^{kh}_{l;j;i}f_hF^2_2
,\label{3}\\
\textbf{B}^k_{i\beta
l}\!\!\!\!&=&\!\!\!\!2fC^{kh}_{i;l}f_hv_\beta
,\\
\textbf{B}^k_{\alpha\beta
l}\!\!\!\!&=&\!\!\!\!2fg_{\alpha\beta}C^{kh}_lf_h,\label{2}\\
\textbf{B}^k_{\alpha\beta\lambda}\!\!\!\!&=&\!\!\!\!-2fC_{\alpha\beta\lambda}f^k,\label{1}\\
\textbf{B}^\gamma_{i\beta\lambda}\!\!\!\!&=&\!\!\!\!\textbf{B}^\gamma_{ij\lambda}
=\textbf{B}^\gamma_{ijk}=0.\label{4}
\end{eqnarray}
\end{lem}
Let $(M_1\times{}_{f}M_2,F)$ is a Berwald manifold. Then we have
$\textbf{B}^d_{abc}=0$. By  using (\ref{1}),  we get
\[
C_{\alpha\beta\lambda}f^k=0.
\]
Multiply this equation in $g_{kr}$ we obtain
\[
C_{\alpha\beta\lambda}f_r=0.
\]
Thus if $f$ is not constant on $M_1$, then we have $C_{\alpha\beta\lambda}=0$. Also, from (\ref{2}) we result that
\[
C^{kh}_lf_h=0.
\]
Differentiating this equation with respect $y^j$ gives us
\[
C^{kh}_{l;j}f_h=0.
\]
Similarly we obtain
\[
C^{kh}_{l;j;i}f_h=0.
\]
Setting the last equation in (\ref{3}) implies that $B^k_{ijl}=0$, i. e., $(M_1, F_1)$ is Berwaldian. These explanations give us the following theorem.
\begin{thm}
Let $(M_1\times_fM_2, F)$ be a twisted product Finsler manifold and $f$ is not constant on $M_1$. Then $(M_1\times_fM_2, F)$ is Berwaldian if and only if $(M_1, F_1)$ is Berwaldian, $(M_2, F_2)$ is Riemannian and the equation $C^{kh}_lf_h=0$ is hold.
\end{thm}

\bigskip

But if $f$ is constant on $M_1$, i.e., $f_i=0$, then we get the following.
\begin{thm}
Let $(M_1\times_fM_2, F)$ be a twisted product Finsler manifold and $f$ is constant on $M_1$. Then $(M_1\times_fM_2,F)$ is Berwaldian if and only if $(M_1, F_1)$ is Berwaldian and the Berwald curvature of $(M_2, F_2)$ satisfies in the following equation:
\begin{eqnarray}
B^\gamma_{\alpha\beta\lambda}\!\!\!\!&=&\!\!\!\!-f^{-1}(C^{\gamma\nu}_{\beta;\alpha;\lambda}f_\nu F_2^2+2C^{\gamma\nu}_{\beta;\alpha}f_\nu v_\lambda+2C^{\gamma\nu}_{\lambda;\alpha}f_\nu v_\beta+2C^{\gamma\nu}_{\alpha}f_\nu g_{\lambda\beta}\nonumber\\
\!\!\!\!&&\!\!\!\!+2C^{\gamma\nu}_{\beta;\lambda}f_\nu v_\alpha+2C^{\gamma\nu}_{\beta}f_\nu g_{\lambda\alpha}+2C^{\gamma\nu}_\lambda f_\nu g_{\alpha\beta}-2g^{\gamma\nu}C_{\alpha\beta\lambda}f_\nu)
\end{eqnarray}
\end{thm}

\bigskip

Here, we consider  twisted product Finsler manifold $(M_1\times_fM_2, F)$ of isotropic Berwald curvature.

\begin{thm}
Every isotropic Berwald twisted product Finsler manifold $(M_1\times_fM_2, F)$ is a Berwald manifold.
\end{thm}
\begin{proof}
Let $(M_1\times_fM_2, F)$ be an isotropic Berwald manifold. Then we have
\[
\textbf{B}^d_{abc}=cF^{-1}\{\textbf{h}^d_a\textbf{h}_{bc}+\textbf{h}^d_b\textbf{h}_{ac}+\textbf{h}^d_c\textbf{h}_{ab}
+2\textbf{C}_{abc}\textbf{y}^d\},
\]
where $c=c(\textbf{x})$ is a function on $M$. Setting $a=j$, $b=k$, $c=l$, $d=\gamma$ and using (\ref{4}) imply that
\[
cF^{-1}\{\frac{3}{F^2}y_jy_ky_lv^\gamma-v^\gamma(y_jg_{kl}+y_kg_{jl}+y_lg_{jk})\}=0.
\]
Multiplying the above equation in $y^jy^k$, we derive that $cf^2F_1^2F_2^2=0$. Thus we have $c=0$, i.e., $(M_1\times_fM_2)$ is Berwaldian.
\end{proof}

\bigskip
Now, we are going to study twisted product Finsler manifold of isotropic mean Berwald curvature. For this work, we must compute  the coefficients of mean Berwald curvature of a twisted product Finsler manifold.

\begin{lem}\label{Lemma2}
Let $(M_1\times_fM_2, F)$ be a twisted product Finsler manifold. Then the coefficients of mean Berwald curvature are as follows:
\begin{eqnarray}
\textbf{E}_{\alpha\beta}\!\!\!\!&=&\!\!\!\!E_{\alpha\beta}+fg_{\alpha\beta}I^hf_h+\frac{1}{2}fI^\nu_{;\alpha;\beta}f_\nu F_2^2
+f^{-1}f_\nu(C^{\gamma\nu}_{\alpha;\beta}v_{\gamma}+I^\nu_{;\alpha}v_\beta+I^\nu_{;\beta}v_\alpha\nonumber\\
\!\!\!\!&&\!\!\!\!+C^\nu_{\alpha\beta}
+I^\nu g_{\alpha\beta}),\label{mean1}\\
\textbf{E}_{ij}\!\!\!\!&=&\!\!\!\!E_{ij}+\frac{1}{2}fI^h_{;j;i}f_hF_2^2,\label{mean2}\\
\textbf{E}_{i\beta
}\!\!\!\!&=&\!\!\!\!fI^h_{;i}f_hv_\beta,\label{mean3}
\end{eqnarray}
where $E_{ij}$ and $E_{\alpha\beta}$ are the coefficients of mean Berwald curvature of $(M_1, F_1)$ and $(M_2, F_2)$, respectively.
\end{lem}
\begin{proof}
By definition and Lemma \ref{Berwald}, we get the proof.
\end{proof}

\bigskip

\begin{thm}
The twisted product Finsler manifold $(M_1\times_fM_2, F)$ is weakly Berwald if and only if $(M_1, F_1)$ is weakly Berwald, $I^hf_h=0$ and the following hold
\be
E_{\alpha\beta}=-\frac{1}{2}fI^\nu_{;\alpha;\beta}f_\nu F_2^2
-f^{-1}f_\nu(C^{\gamma\nu}_{\alpha;\beta}v_{\gamma}+I^\nu_{;\alpha}v_\beta+I^\nu_{;\beta}v_\alpha
+C^\nu_{\alpha\beta}
+I^\nu g_{\alpha\beta}).\label{mean4}
\ee
\end{thm}
\begin{proof}
If $(M_1\times_fM_2)$ be a weakly Berwald manifold, then we have
\[
\textbf{E}_{\alpha\beta}=\textbf{E}_{ij}=\textbf{E}_{i\beta}=0.
\]
Thus by using (\ref{mean3}) we result that $I^h_{;i}f_h=0$. This equation implies that
\[
I^h_{;j;i}f_h=0, \ \ \  \textrm{and} \ \ \ I^hf_h=0.
\]
By setting these equations in (\ref{mean1}) and (\ref{mean2}) we conclude that $E_{ij}=0$ and $E_{\alpha\beta}$ satisfies in (\ref{mean4}).
\end{proof}

\bigskip

Now, if $f$ is constant on $M_2$, then (\ref{mean4}) implies that $E_{\alpha\beta}=0$. Thus we conclude the following.

\begin{cor}
Let $(M_1\times_fM_2, F)$ be a twisted product Finsler manifold and $f$ is a function on $M_1$, only. Then $(M_1\times_fM_2, F)$ is weakly Berwald if and only if
 $(M_1, F_1)$ and $(M_2, F_2)$ are weakly Berwald manifolds and $I^hf_h=0$.
\end{cor}
Now, we consider twisted product Finsler manifolds with  isotropic mean
Berwald  curvature. It is remarkable that as a consequence of Lemma \ref{Lemma2},
we have the following.
\begin{lem}
Twisted product Finsler manifold $(M_1\times_{f}M_2, F)$ is
isotropic mean Berwald manifold if and only if
\begin{eqnarray}
E_{\alpha\beta}\!\!\!\!&+&\!\!\!\!fg_{\alpha\beta}I^hf_h+\frac{1}{2}fI^\nu_{;\alpha;\beta}f_\nu F_2^2
+f^{-1}f_\nu(C^{\gamma\nu}_{\alpha;\beta}v_{\gamma}+I^\nu_{;\alpha}v_\beta+I^\nu_{;\beta}v_\alpha\nonumber\\
\!\!\!\!&&\!\!\!\!+C^\nu_{\alpha\beta}
+I^\nu g_{\alpha\beta})-\frac{n+1}{2}cf^2F^{-1}(g_{\alpha\beta}-\frac{f^2}{F^2}v_\alpha v_\beta)=0,\label{wmean1}\\
E_{ij}\!\!\!\!&+&\!\!\!\!\frac{1}{2}fI^h_{;j;i}f_hF_2^2-\frac{n+1}{2}cF^{-1}(g_{ij}-\frac{1}{F^2}y_iy_j)=0,\label{wmean2}\\
\!\!\!\!&&\!\!\!\!c(n+1)F^{-3}y_i+fI^h_{;i}f_h=0,\label{wmean3}
\end{eqnarray}
where $c=c(\textbf{x})$ is a scalar function on $M$.
\end{lem}

\bigskip

\begin{thm}
Every twisted product Finsler manifold $(M_1\times_{f}M_2, F)$
with isotropic mean Berwald curvature is a weakly Berwald
manifold.
\end{thm}
\begin{proof}
Suppose that $F$ is isotropic mean Berwald twisted product Finsler
metric. Then differentiating (\ref{wmean3}) with respect
$v^\gamma$ gives us
\[
c(n+1)f^2F^{-5}v_\gamma y_i=0
\]
Thus, we conclude that $c=0$. This implies that $F$  reduces to a  weakly
Berwald metric.
\end{proof}

\section{Locally Dually Flat Twisted Product Finsler Manifolds}
In \cite{amna}, Amari-Nagaoka introduced the notion of dually flat Riemannian metrics when they study the information geometry on Riemannian manifolds. Information geometry has emerged from investigating the geometrical structure of a family of probability distributions and has been applied successfully to various areas including statistical inference, control system theory and multi-terminal information theory. In Finsler geometry, Shen extends the notion of locally dually flatness for Finsler metrics \cite{shenLec}. Dually flat Finsler metrics form a special and valuable class of Finsler metrics in Finsler information geometry, which play a very   important role in studying flat Finsler information structure \cite{CSZ}\cite{X2}.

In this section, we study locally dually flat  twisted product
Finsler metrics. It is remarkable that, a Finsler metric
$F=F(\textbf{x},\textbf{y})$ on a manifold $M$ is said to be
locally dually flat if at any point there is a standard coordinate
system
 $(\textbf{x}^a,\textbf{y}^a)$ in $TM$ such that it  satisfies
\be \frac{\partial^2 F^2}{\partial \textbf{x}^b \partial
\textbf{y}^a}\ \textbf{y}^b=2\ \frac{\partial F^2}{\partial
\textbf{x}^a}.\label{100} \ee In this case, the coordinate
$(\textbf{x}^a)$ is called an adapted local coordinate system. By
using (\ref{100}) we can obtain the following lemma.
\begin{lem}
Let $(M_1\times_{f}M_2, F)$ be a twisted product Finsler manifold.
Then $F$ is locally dually flat if and only if $F_1$ and $F_2$
satisfy in the following equations
\begin{eqnarray}
&&\frac{\partial^2F_1^2}{\partial x^k\partial
y^l}y^k=2\frac{\partial F_1^2}{\partial
x^l}+4ff_lF_2^2,\label{Im1}\\
&&4f_kv_\beta y^k+f\frac{\partial^2F_2^2}{\partial u^\alpha\partial
v^\beta}v^\alpha+4f_\alpha v_\beta v^\alpha= 2f\frac{\partial
F_2^2}{\partial u^\beta}+4f_\beta F_2^2. \label{Im2}
\end{eqnarray}
\end{lem}

Now, let $F$ be a locally dually flat Finsler metric. Taking
derivative with respect to $v^{\gamma}$ from (\ref{Im1}) yields $f_l=0$, which means that $f$ is a constant function on $M_1$. In
this case, the relations (\ref{Im1}) and (\ref{Im2}) reduce to the following
\begin{equation}
\frac{\partial^2F_1^2}{\partial x^k\partial
y^l}y^k=2\frac{\partial F_1^2}{\partial x^l},\label{Im3}
\end{equation}
\begin{equation}
f\frac{\partial^2F_2^2}{\partial u^\alpha\partial
v^\beta}v^\alpha+4f_\alpha v_\beta v^\alpha=2f\frac{\partial
F_2^2}{\partial u^\beta}+4f_\beta F_2^2. \label{Im4}
\end{equation}
By (\ref{Im3}),  we deduce that $F_1$ is locally dually flat.

Now, we assume that $F_1$ and $F_2$ are locally dually flat Finsler
metrics. Then we have
\begin{eqnarray}
&&\frac{\partial^2F_1^2}{\partial x^k\partial
y^l}y^k=2\frac{\partial F_1^2}{\partial x^l},\label{IM115}\\
&&\frac{\partial^2F_2^2}{\partial u^\alpha\partial
v^\beta}v^\alpha=2\frac{\partial F_2^2}{\partial
u^\beta}.\label{Im5}
\end{eqnarray}
By (\ref{IM115}) and (\ref{Im5}),  we derive that (\ref{Im1}) and (\ref{Im2})
are hold if and only if the following hold
\begin{equation}
f_l=0, \ \ \ \  \ f_\alpha v_\beta v^\alpha=f_\beta F_2^2. \label{Im6}
\end{equation}
Therefore we can conclude the following.
\begin{thm}\label{THM1}
Let $(M_1\times_{f}M_2, F)$ be a twisted product Finsler
manifold.\\
(i)\ If $F$ is locally dually flat then $F_1$ is locally dually
flat, $f$ is a function with respect $(u^\alpha)$ only and $F_2$
satisfy in (\ref{Im4}).\\
(ii) If $F_1$ and $F_2$ are locally dually flat. Then $F$ is
locally dually flat if and only if $f$ is a function with respect
$(u^\alpha)$ only and $F_2$ satisfies in (\ref{Im6}).
\end{thm}

\bigskip

By Theorem \ref{THM1}, we conclude the following.
\begin{cor}
There is not exist  any locally dually flat proper twisted product Finsler manifold.
\end{cor}

\noindent
Esmail Peyghan and Leila Nourmohammadi Far\\
Faculty  of Science, Department of Mathematics\\
Arak University\\
Arak 38156-8-8349,  Iran\\
Email: epeyghan@gmail.com

\bigskip

\noindent
Akbar Tayebi\\
Faculty  of Science, Department of Mathematics\\
University of Qom \\
Qom, Iran\\
Email: akbar.tayebi@gmail.com

\end{document}